\documentclass{article}
\usepackage{amsmath}
\usepackage{amsthm}
\usepackage{amsfonts}
\usepackage{amssymb}
\usepackage{amsrefs}
\usepackage{amsthm}
\usepackage{array}
\usepackage{hyperref}

\newcounter{dummy}

\newtheorem{theorem}[dummy]{Theorem}
\newtheorem*{theorem*}{Theorem}

\newtheorem{lemma}[dummy]{Lemma}
\newtheorem{cor}[dummy]{Corollary}
\newtheorem*{cor*}{Corollary}
\newtheorem{prop}[dummy]{Proposition}
\newtheorem*{defn}{Definition}
\theoremstyle{remark}

\newtheorem*{remark*}{Remark}

\numberwithin{equation}{section}

\DeclareMathOperator{\Lk}{Lk}

\begin{document}
\title{The $e$-vector of a simplicial complex}
\author{Wayne A. Johnson\footnote{Department of Mathematics, Truman State University, Kirksville, MO} \and Wiktor J. Mogilski\footnote{Department of Mathematical Sciences, Utah Valley University, Orem, UT}}

\maketitle

\begin{abstract}
We study the exponential Hilbert series (both coarsely- and finely-graded) of the Stanley-Reisner ring of an abstract simplicial complex, $\Delta$, and we introduce the $e$-vector of $\Delta$, which relates to the coefficients of the exponential Hilbert series. We explore the relationship of the $e$-vector with the classical $f$-vector and $h$-vector of $\Delta$. We then prove a simple combinatorial identity that explicitly computes the $e$-vector in the case where $\Delta$ is an Eulerian manifold. This identity leads to an elementary proof of the classical Dehn-Sommerville relations. We conclude by applying the $e$-vector to the computation of the dimension of certain probability models.
\vspace{1 pc}
\end{abstract}

\section{Introduction}

The inspiration for this article comes from work done by the first author in \cite{Joh}. In the same vein, we study the exponential variant of the classical Hilbert series of a graded $S$-module, where $S=k[x_1,\dots,x_n]$ is the polynomial ring (throughout the paper, $S$ will always denote this polynomial ring).

Recall that for any $\mathbb{N}^n$-graded $S$-module, $M$, we define the \emph{Hilbert series} of $M$ to be the formal power series
\begin{center}
$HS(M;x_1,\dots,x_n)=\displaystyle\sum_{\textbf{a}\in\mathbb{N}^n}\dim_k(M_\textbf{a})\textbf{x}^\textbf{a}$,
\end{center}
where, if $\textbf{a}=(a_1,\dots,a_n)$, we define $\textbf{x}^\textbf{a}=x_1^{a_1}x_2^{a_2}\dots x_n^{a_n}$. Note that the formal power series $HS(M;x_1,\dots,x_n)$ is often called the \textit{fine} or \textit{finely-graded} Hilbert series of $M$. Setting each $x_i=t$, we obtain a formal power series in one variable, $HS(M;t)$, which is often called the \textit{coarse} or \textit{coarsely-graded} Hilbert series of $M$.

While the Hilbert series has been extensively studied in various contexts, we will be concerned with an analogue of $HS(M;x_1,\dots,x_n)$ that has been studied considerably less thoroughly. As in \cite{Joh}, define the \textit{exponential Hilbert series} of $M$ to be the formal power series
\begin{center}
$E(M;x_1,\dots,x_n)=\displaystyle\sum_{\textbf{a}\in\mathbb{N}^n}\dim_k(M_\textbf{a})\frac{\textbf{x}^\textbf{a}}{\textbf{a}!}$,
\end{center}
where $\textbf{a}!:=a_1!a_2!\cdot\cdot\cdot a_n!$. In similar fashion to $HS(M;x_1,\dots,x_n)$, we consider the exponential series in fine and coarse flavors. If we wish to equate the variables to obtain the coarse exponential series, we denote this by $E(M;t)$.

The featured graded $S$-module in this article is the \textit{Stanley-Reisner ring}. Let $\Delta$ be an abstract simplicial complex. Recall (see, for example, \cite{MiS}) that the \textit{Stanley-Reisner ideal} of $\Delta$, denoted $I_\Delta$ is defined to be the ideal generated by the monomials corresponding to the non-faces of $\Delta$. The \textit{Stanley-Reisner ring} is defined to be the quotient ring $S/I_\Delta$, where $\{1,2,\dots,n\}$ is the vertex set of $\Delta$. As any ideal generated by monomials will automatically be graded, the Stanley-Reisner ring inherits an $\mathbb{N}^n$-gradation from $S$.

The algebraic and combinatorial properties of the Stanley-Reisner ring of a simplicial complex have been studied extensively (see \cite{MiS}, \cite{Stan96} for broad treatments). Our first main contribution is the explicit computation of the fine exponential series of $S/I_\Delta$.

\begin{theorem*}[Theorem 2]
Let $\Delta$ be an abstract simplicial complex with Stanley-Reisner ideal $I_\Delta$. Then
\begin{center}
$E(S/I_\Delta;x_1,\dots,x_n)=\displaystyle\sum_{\sigma\in\Delta}\prod_{i\in\sigma}\left(e^{x_i}-1\right)$.
\end{center}
\end{theorem*}

Setting $x_i=t$ for all $i$, the coarse exponential series simply becomes a polynomial in $e^t$:
$$
E(S/I_\Delta;t)=\sum_{\sigma\in\Delta}(e^t-1)^{\dim(\sigma)+1}=e_0+e_1e^t+e_2e^{2t}+\dots+e_de^{dt}.
$$

We call the vector $(e_0,e_1,\dots,e_d)$ consisting of the coefficients of the above the $e$-\textit{vector} of $\Delta$.

In contrast, the ordinary coarse Hilbert series is computed to be (see \cite{MiS} or \cite{Stan96}):

\begin{center}
$HS(S/I_\Delta;t)=\displaystyle\frac{K(t)}{(1-t)^d}$,
\end{center}
where $K(t)$ is a polynomial (of degree $d$) with integer coefficients. If we write the polynomial as
\begin{center}
$K(t)=h_0+h_1t+h_2t^2+\dots+h_dt^d$,
\end{center}
then the vector $(h_0,h_1,\dots,h_d)$ consisting of the coefficients of $K(t)$ is called the $h$-\emph{vector} of $\Delta$, which is a classical invariant of $\Delta$.

Associated to an abstract simplicial complex there is one more classical invariant. The $f$-\emph{vector}, also called the \textit{face vector} of $\Delta$ is usually written as $(f_{-1},f_0,\dots,f_{d-1})$, where $f_i$ give the number of $i$-faces of $\Delta$. For algebraic purposes it is common to count the empty set, and hence we set $f_{-1}=1$.

The relationship between the $f$-vector and the $h$-vector was a topic of extensive study in the 1960s, 70s, and 80s, arguably culminating in the proofs of the Upper Bound Conjecture in \cite{Stan75} of Stanley, and the $g$-Theorem in \cite{Stan80} of Stanley and \cite{BL80} \cite{BL81} of Billera and Lee. The results of these papers and other useful results pertaining to the $f$-vector and the $h$-vector are collected in Chapter II of \cite{Stan96}. For a more modern treatment, including a survey of recent results pertaining to the $f$-vector and $h$-vector, see the survey by Klee and Novik (\cite{KN}). The relationship between all of the above vectors is presented in equations (2.5) and (2.6) below.

A part of the $g$-Theorem guarantees that the $h$-vector is palindromic (satisfying $h_k=h_{d-k}$) whenever $\Delta$ is the boundary complex of a simplicial polytope (or more generally, an Eulerian sphere). These equations are referred to as the \emph{Dehn-Sommerville relations}. In \cite{Stan96}, it is shown that these equations relate to the intersection cohomology of a certain toric variety associated to the polytope. Thus the $h$-vector is not only more efficient, but also encodes more data about $\Delta$.

Our main application is the surprising fact that the Dehn-Sommerville relations have a very simple characterization in terms of the $e$-vector. We say that a $(d-1)$-dimensional simplicial complex $\Delta$ has \textit{Property E} if for every $0\leq k\leq d$ we have that $e_k=(-1)^{d-k} f_{k-1}$ (in other words, the $e$-vector is an alternating version of the $f$-vector). If these equalities are just required to hold for $1\leq k\leq d$, then we say that $\Delta$ has \textit{weak Property E}.

\begin{theorem*}[Theorem 3]
A $(d-1)$-simplicial complex $\Delta$ has weak Property E if and only if for every $0\leq k\leq d$ we have that $$h_k-h_{d-k}=(-1)^k {d\choose k}\left(\tilde{\chi}(S^{d-1})-\tilde{\chi}(\Delta) \right).$$

In particular, $\Delta$ has Property E if and only if for every $0\leq k\leq d$ we have that $h_k=h_{d-k}$.
\end{theorem*}

This characterization using the $e$-vector has the following nice applications. First, it leads to an elementary proof of the Dehn-Sommerville relations for Eulerian manifolds.

\begin{theorem*}[Theorem 5]
If $\Delta$ is an Eulerian manifold then $\Delta$ has weak Property E. If $\Delta$ is further assumed to have odd dimension or is an Eulerian sphere, then $\Delta$ has Property E.
\end{theorem*}

Simply put, Eulerian manifolds are just simplicial complexes that are modelled on triangulations of manifolds. Theorem 5 gives us an explicit computation for the $e$-vectors of such complexes, and, in turn, explicit computation of the exponential Hilbert series. For our second application, we are able to easily produce examples of non-Eulerian complexes satisfying Property E in every dimension, thus showing that the converse of Theorem 5 is false. Hence, via Theorem 3, we produce a class of non-Eulerian examples satisfying the Dehn-Sommerville equations. We then consider a lesser known version of the Dehn-Sommerville relations for Eulerian manifolds with boundary and provide an equivalent formulation using the $e$-vector (Theorem 8).

We conclude with an application of the $e$-vector to algebraic statistics, where it is used to compute the dimension of a hierarchical log-linear model. This dimension was original computed by Ho\c{s}ten and Sullivant in \cite{HS}. The connection to the exponential Hilbert series of the complex that describes the independence relations in the model is new.

\section{Exponential Hilbert series and the $e$-vector}

We begin with a broader discussion of the exponential Hilbert series of certain $S$-modules. First, note that if $M=S$, we have
\begin{center}
$E(S;x_1,\dots,x_n)=e^{x_1}e^{x_2}\dots e^{x_n}=\textbf{e}^\textbf{x}$.
\end{center}
The $S$-modules we will chiefly be concerned with are those generated by a set of fixed monomials. We begin with the simple case of the free $S$-module generated by $\textbf{x}^\textbf{a}$. Following the notation in \cite{MiS}, denote such a module by $S(-\textbf{a})$. In analogy to the ordinary Hilbert series, we have the following result.
\begin{prop}
The exponential Hilbert series of $S(-\textbf{a})$ has the form
\begin{center}
$E(S(-\textbf{a});x_1,\dots,x_n)=\displaystyle\prod_{i=1}^n\left(e^{x_i}-\sum_{k=0}^{a_i-1}\frac{x_i^k}{k!}\right)$.
\end{center}
\end{prop}
\begin{proof}
First, note that $S(-\textbf{a})=S\cdot\textbf{x}^\textbf{a}$, and we then have
\begin{center}
$E(S(-\textbf{a});x_1,\dots,x_n)=\displaystyle\sum_{\textbf{b}\in\mathbb{N}^n}\frac{\textbf{x}^{\textbf{a}+\textbf{b}}}{(\textbf{a}+\textbf{b})!}=\textbf{x}^\textbf{a}\sum_{\textbf{b}\in\mathbb{N}^n}\frac{\textbf{x}^{\textbf{b}}}{(\textbf{a}+\textbf{b})!}$.
\end{center}
In order to compute the series on the right, we first compute the univariate version of the series. If there is only one variable, then we have
\begin{center}
$\displaystyle\sum_{k\in\mathbb{N}}\frac{x^k}{(k+d)!}=\frac{e^x-\sum_{j=0}^{d-1}\frac{x^j}{j!}}{x^d}$.
\end{center}
With that in mind, our original series becomes
\begin{center}
$E(S(-\textbf{a});x_1,\dots,x_n)=\textbf{x}^\textbf{a}\displaystyle\prod_{i=1}^n\frac{e^{x_i}-\sum_{j=0}^{a_i-1}\frac{x_i^j}{j!}}{x_i^{a_i}}$,
\end{center}
and distributing the $\textbf{x}^\textbf{a}$ yields the result.
\end{proof}
The remainder of this section is devoted to expanding this result to the Stanley-Reisner ring, $S/I_\Delta$, of a simplicial complex, $\Delta$. While the argument in this case is similar to the above, the structure of $I_\Delta$ greatly simplifies the formula of the exponential Hilbert series, as the monomials are square-free.

\begin{theorem}
Let $\Delta$ be an abstract simplicial complex with Stanley-Reisner ideal $I_\Delta$. Then
\begin{center}
$E(S/I_\Delta;x_1,\dots,x_n)=\displaystyle\sum_{\sigma\in\Delta}\prod_{i\in\sigma}\left(e^{x_i}-1\right)$.
\end{center}
\end{theorem}

\begin{proof}
The proof is adapted from that of the ordinary Hilbert series found in \cite{MiS}. First, note that $\textbf{x}^\textbf{a}$ lies outside of $I_\Delta$ if and only if $\textbf{x}^{\text{supp}(\textbf{a})}$ lies outside of $I_\Delta$, where $\text{supp}(\textbf{a})=\{i\mid a_i\neq0\}$ is called the \textit{support} of $\textbf{a}$. Then, by the definition of the Stanley-Reisner ideal, the monomials that do not vanish in $S/I_\Delta$ are precisely those whose support is in $\Delta$. The result follows from the computation in Proposition 1 and the fact that the monomials $x^{\text{supp}(\textbf{a})}$ are square-free.
\end{proof}

This theorem begins to exhibit one of the chief differences between the ordinary Hilbert series and its exponential counterpart. In the ordinary case (see \cite{MiS}), the Hilbert series of a Stanley-Reisner ring can be written as
\begin{center}
$\displaystyle\frac{K(x_1,\dots,x_n)}{(1-x_1)(1-x_2)\dots(1-x_n)}$,
\end{center}
where $K(x_1,\dots,x_n)$ is a polynomial with integer coefficients (often called the $K$-\emph{polynomial} of $S/I_\Delta$). However, the exponential series relates to a different polynomial. Namely, the exponential Hilbert series of $S/I_\Delta$ is polynomial in the $e^{x_i}$. We may also note, as seen in Proposition 1, that the exponential Hilbert series of $S(-\textbf{a})$ is polynomial in the $x_i$ and the $e^{x_i}$. In either case, the exponential Hilbert series seems to converge to a polynomial outright, rather than in relation to a rational function.

We now turn our focus to the coarsely-graded exponential Hilbert series obtained by setting $x_i=t$ for all $i$. The computation of Theorem 2 implies that
\begin{equation}
E(S/I_\Delta;t)=\displaystyle\sum_{\sigma\in\Delta}(e^t-1)^{|\sigma|}=\sum_{\sigma\in\Delta}(e^t-1)^{\dim(\sigma)+1}.
\end{equation}
Note again that this is polynomial in the $e^t$ and the constant term is
\begin{equation}
1-f_0+f_1-f_2+\dots=-\chi(\Delta),
\end{equation}
where $\chi(\Delta)$ is the reduced Euler characteristic of $\Delta$. Note that the number of terms with power $|\sigma|$ is equal to the number of faces of $\Delta$ with dimension equal to that of $\sigma$, including the empty face. Hence we can rewrite (2.1) as follows:
\begin{equation}
E(S/I_\Delta;t)=\displaystyle\sum_{i=0}^d f_{i-1}(e^t-1)^i,
\end{equation}
where $d:=\dim(\Delta)+1$.

Upon expansion, we get a polynomial of degree $d$ in $e^t$:
\begin{equation}
E(S/I_\Delta;t)=e_0+e_1e^t+e_2e^{2t}+\dots+e_de^{dt}.
\end{equation}
We call the vector $(e_0,e_1,\dots,e_d)$ consisting of the coefficients of the polynomial in (2.4) the $e$-\textit{vector} of $\Delta$. We have already shown that $e_0=-\chi(\Delta)$. We next explore the remaining coefficients. Expanding $E(S/I_\Delta;t)$ yields
\begin{center}
$\displaystyle\sum_{i=0}^df_{i-1}(e^t-1)^i=\sum_{i=0}^df_{i-1}\sum_{j=0}^i(-1)^{i-j}{i\choose j}e^{jt}$.
\end{center}
We wish to isolate the coefficient of $e^{kt}$ for some fixed $k\in\{0,\dots,d\}$. The only values of $i$ which contribute a term to this coefficient are those where $i\geq k$. Therefore, the coefficient of $e^{kt}$ is:
\begin{equation}
e_k=\displaystyle\sum_{i=k}^d(-1)^{i-k}f_{i-1}{i\choose k}.
\end{equation}
Note that there is an analogous formula for the components of the $h$-vector of $\Delta$ (see \cite{Stan96}), given below:
\begin{equation}
h_k=\displaystyle\sum_{i=0}^k(-1)^{k-i}f_{i-1}{d-i\choose k-i}.
\end{equation}
\vspace{1pc}
\\
\noindent\textbf{Example:} Let
\begin{center}
$\Delta=\{\{1\},\{2\},\{3\},\{4\},\{1,2\},\{1,3\},\{2,3\},\{2,4\},\{3,4\},\{1,2,3\}\}$.
\end{center}
Then the $f$-vector is given by $(1,4,5,1)$. Using (2.5) and (2.6), we can easily compute the $e$-vector $(1,-3,2,1)$ and the $h$-vector $(1,1,0,-1)$.

\section{The $e$-vector and Property E}

We begin this section with a motivating example.
\vspace{1pc}
\\
\noindent\textbf{Example:} Let $\Delta$ be the boundary of the tetrahedron. This complex has $f$-vector $(1,4,6,4)$, $e$-vector $(-1,4,-6,4)$ and $h$-vector $(1,1,1,1)$. Note that the $e$-vector is an alternating version of the $f$-vector. This is no accident, as we will see in this section.
\vspace{1pc}
\\
\indent Fixing a $(d-1)$-dimensional simplicial complex $\Delta$, we denote the $e$-components, $f$-components, and $h$-components of $\Delta$ as $e_i(\Delta)$, $f_{i-1}(\Delta)$, and $h_i(\Delta)$, respectively. We then define the following three polynomials. The \textit{$f$-polynomial} is defined to be $$f_\Delta(t)=\sum_{i=0}^d f_{i-1}(\Delta) t^i,$$ the \textit{$e$-polynomial} is defined to be $$e_\Delta(t)=\sum_{i=0}^d e_i(\Delta) t^i,$$ and finally the \textit{$h$-polynomial} is defined to be $$h_\Delta(t)=\sum_{i=0}^d h_i(\Delta) t^i.$$

If the simplicial complex $\Delta$ is understood, then we sometimes omit the subscript. The relationship between the $f$-polynomial and $h$-polynomial is classically given by $$h_\Delta(t)=(1-t)^df_\Delta\left(\frac{t}{1-t}\right).$$ From equations (2.3) and (2.4), it follows that $e_\Delta(t)=f_\Delta(t-1)$, and hence $$h_\Delta(t)=(1-t)^d e_\Delta\left(\frac{1}{1-t} \right).$$

Note that $e_\Delta(0)=e_0=-\chi(\Delta)$ and $e_\Delta(1)=f_\Delta(0)=f_{-1}=1$. This implies that $$\sum_{i=1}^d e_i=1+\chi(\Delta)=\sum_{i=0}^{d-1} (-1)^i f_i=\tilde{\chi}(\Delta),$$ where $\tilde{\chi}(\Delta)$ is the ordinary topological Euler characteristic of $\Delta$ (ignoring the empty face). Combining this observation with the above example motivates the following definition.

\begin{defn} We say that a $(d-1)$-dimensional simplicial complex $\Delta$ has \textit{Property E} if for every $0\leq k\leq d$ we have that $e_k=(-1)^{d-k} f_{k-1}$. If these equalities are just required to hold for $1\leq k\leq d$, then we say that $\Delta$ has \textit{weak Property E}.
\end{defn}

\begin{remark*}
Note that, in terms of the $e$-polynomial, Property E just says that $e_\Delta(t)=(-1)^df_\Delta(-t)$, while weak Property E is equivalent to the equation $e_\Delta(t)+(-1)^{d+1}f_\Delta(-t)=e_0+(-1)^{d+1}$. Also note that Property E implies that $e_0=-\chi(\Delta)=(-1)^d$, which in turn implies that $\Delta$ has the topological Euler characteristic of a $(d-1)$-sphere. Hence Property E is very restrictive. We introduce the notion of weak Property E to allow for more simplicial complexes.
\end{remark*}
\vspace{1pc}
\noindent\textbf{Example:} Let $\Delta$ be a two-dimensional simplicial complex. Then the $e$-vector has the following formula in terms of the face numbers of $\Delta$:

$$\left(1-f_0+f_1-f_2,f_0-2f_1+3f_2,f_1-3f_2,f_2 \right).$$

If $\Delta$ satisfies Property E, then we obtain the following equations relating the face numbers of $\Delta$: $$f_0-f_1+f_2=2,$$ $$3f_2=2f_1.$$

These are familiar equations satisfied by the boundary of a simplicial polyhedron. The latter equation simply says that each edge of the simplicial polyhedron is contained in two triangles and that each triangle contains three edges. Thus, in general, one can think of the $e$-vector as measuring how much $\Delta$ deviates from satisfying the standard face equations of the boundary of a simplicial polytope.
\vspace{1pc}
\\
\indent We now prove the main theorem of this section, which says that (weak) Property E on the $e$-vector is equivalent to the (general) Dehn-Sommerville condition on the $h$-vector. Recall that if the $h$-vector satisfies the classical Dehn-Sommerville equations, then $h_k=h_{d-k}$ (so, in particular, $h$ is symmetric). The $h$-vector satisfies the general Dehn-Sommerville equations if
\begin{equation}
h_k-h_{d-k}=(-1)^k {d\choose k}\left(\tilde{\chi}(S^{d-1})-\tilde{\chi}(\Delta) \right),
\end{equation}
where $\tilde{\chi}(S^{d-1})=1+(-1)^{d-1}$ is just the Euler characteristic of the $(d-1)$-sphere.

\begin{remark*} The general Dehn-Sommerville equations are sometimes referred to as Klee's Dehn-Sommerville equations, as they were originally considered by V. Klee in \cite{Kle}.\end{remark*}

\begin{theorem}
A $(d-1)$-simplicial complex $\Delta$ has weak Property E if and only if for every $0\leq k\leq d$ we have that $$h_k-h_{d-k}=(-1)^k {d\choose k}\left(\tilde{\chi}(S^{d-1})-\tilde{\chi}(\Delta) \right).$$

In particular, $\Delta$ has Property E if and only if for every $0\leq k\leq d$ we have that $h_k=h_{d-k}$.
\end{theorem}

\begin{proof}
We restate the conditions of the theorem in terms of polynomials. Recall that a simplicial complex has weak Property E if and only if $$e_\Delta(t)+(-1)^{d+1}f_\Delta(-t)=e_\Delta(t)+(-1)^{d+1} e(1-t)=e_0+(-1)^{d+1}.$$

In terms of the $h$-polynomial, the generalized Dehn-Sommerville condition on the $h$-vector in the theorem is equivalent to $$h(t)-t^d h\left(\frac{1}{t}\right)=C\cdot K(t),$$

where $C=1-\tilde{\chi}(\Delta)+(-1)^{d+1}=e_0+(-1)^{d+1}$ and $$K(t)=\sum_{k=0}^d (-1)^k {d \choose k} t^k. $$

In terms of the $e$-polynomial, the $h$-polynomial equation can be rewritten as $$(1-t)^de\left(\frac{1}{1-t}\right)-(t-1)^d e\left(\frac{t}{t-1}\right)=C\cdot K(t).$$

Factoring $(1-t)^d$ from the left of the equation and moving it to the right-hand side, the above equation becomes:

$$e\left(\frac{1}{1-t}\right)+(-1)^{d+1} e\left(\frac{t}{t-1}\right)=C\cdot \frac{K(t)}{(1-t)^d}.$$

Making the substitution $t=1-\frac{1}{u}$, we obtain $$e(u)+(-1)^{d+1} e(1-u)=C\cdot u^d K\left(\frac{u-1}{u} \right).$$

We have that $u^d K\left(\frac{u-1}{u} \right)=1$, as $$u^d K\left(\frac{u-1}{u} \right)=\sum_{k=0}^d (-1)^k {d\choose k} \left(\frac{u-1}{u} \right)^k u^d=\sum_{k=0}^d {d\choose k} (1-u)^k u^{d-k}=1 .$$

Note that the right-most equality is an application of the Binomial Theorem. Thus our equation of interest finally becomes $$e(u)+(-1)^{d+1} e(1-u)=C=e_0+(-1)^{d+1}.$$

This computation shows that the general Dehn-Sommerville condition on the $h$-vector is equivalent to weak Property E, and completes the proof of the first part of the theorem.

The second part of the theorem now follows from the first. Note that if $\Delta$ has Property E, then it has weak Property E. Furthermore, Property E implies that $\tilde{\chi}(\Delta)=\tilde{\chi}(S^{d-1})$ (see the previous remark), and hence for every $0\leq k\leq d$ we have that $h_k=h_{d-k}$. If we assume that for every $0\leq k\leq d$ we have that $h_k=h_{d-k}$, then $\tilde{\chi}(\Delta)=\tilde{\chi}(S^{d-1})$ and $\Delta$ has weak Property E, and thus $\Delta$ has Property E.
\end{proof}

\begin{remark*}
An immediate consequence of the above theorem is that a Dehn-Sommerville space has zero topological Euler characteristic.
\end{remark*}

As an application, we will exploit Theorem 3 to provide an elementary proof of the general Dehn-Sommerville equations. But first we need to recall a definition and prove a lemma. For a simplex $\sigma\in\Delta$ we define the \emph{link of $\sigma$} to be the following subcomplex: $$\Lk\sigma:=\{\tau\in\Delta\mid \sigma\cup\tau\in\Delta \text{ and } \sigma\cap\tau=\emptyset\}.$$

The following lemma gives a local check to verify if a simplicial complex has weak Property E.

\begin{lemma}
Suppose that for every $v\in \text{vert}(\Delta)$ we have that $\Lk v$ has Property E. Then $$e_\Delta(t)+(-1)^{d+1}f_\Delta(-t)=e_0+(-1)^{d+1} .$$ In particular, $\Delta$ has weak Property E.
\end{lemma}
\begin{proof} Since every $j$-dimensional simplex of $\Delta$ has $j+1$ vertices, it follows that $$(j+1)f_j(\Delta)=\sum_{v\in \text{vert}(\Delta)} f_{j-1}(\Lk v).$$ Rewriting this in terms of the $f$-polynomial, we have that $$\frac{d}{dt} f_\Delta(t)=\sum_{v\in \text{vert}(\Delta)} f_{\Lk v}(t).$$

As $e_\Delta(t)=f_\Delta(t-1)$, it follows that the $e$-polynomial has the same property:

$$\frac{d}{dt}e_\Delta(t)=\frac{d}{dt} f_\Delta(t-1)=\sum_{v\in \text{vert}(\Delta)} f_{\Lk v}(t-1)=\sum_{v\in \text{vert}(\Delta)} e_{\Lk v}(t).$$

Now, suppose that for every $v\in \text{vert}(\Delta)$ we have that $\Lk v$ has Property E. In other words, $$e_{\Lk v}(t)=(-1)^{d-1} f_{\Lk v}(-t).$$

Hence $$\frac{d}{dt}e_\Delta(t)=\sum_{v\in \text{vert}(\Delta)} e_{\Lk v}(t)=\sum_{v\in \text{vert}(\Delta)} (-1)^{d-1} f_{\Lk v}(-t)=(-1)^d \frac{d}{dt} f_\Delta(-t).$$

This implies that $e_\Delta(t)=(-1)^d f_\Delta(-t)+C$. Plugging in $t=0$ to solve for the constant, we obtain $C=(-1)^{d+1}f_{-1}+e_0=(-1)^{d+1}-\chi(\Delta)$, and the result follows.

\end{proof}

We now need a few more definitions. A simplicial complex is \textit{pure} if every maximal face has the same dimension. A pure $(d-1)$-simplicial complex $\Delta$ is an \textit{Eulerian manifold} if for every simplex (except the empty simplex) $\sigma\in \Delta$ we have that $\tilde\chi(\Lk \sigma)=1+(-1)^{d-\dim\sigma}$. Note that this definition just requires that the link of every face (except the empty face) of $\Delta$ has the same Euler characteristic as a sphere of appropriate dimension. If an Eulerian manifold $\Delta$ additionally satisfies $\tilde\chi(\Delta)=1+(-1)^{d-1}$, then $\Delta$ is an \textit{Eulerian sphere}.

The following theorem provides a short proof of the fact that an Eulerian manifold satisfies weak Property E. Combined with Theorem 3, this gives an elementary independent proof of the (general) Dehn-Sommerville Equations. Furthermore, it provides a plethora of examples of complexes satisfying Property E and, in turn, complexes for which the exponential Hilbert series is explicitly computed. For example, it follows that if $\Delta$ is the boundary of a simplicial polytope or is a triangulation of your favorite odd-dimensional manifold, then $\Delta$ has Property E.

\begin{theorem}
If $\Delta$ is an Eulerian manifold then $\Delta$ has weak Property E. If $\Delta$ is further assumed to have odd dimension or is an Eulerian sphere, then $\Delta$ has Property E.
\end{theorem}
\begin{proof}
We begin by noting the standard fact that a complex $\Delta$ is an Eulerian manifold if and only if the link of every vertex is an Eulerian sphere. Hence, if we prove the theorem for Eulerian spheres, then the theorem follows for Eulerian manifolds by Lemma 4.

Suppose that $\Delta^{d-1}$ is a Eulerian sphere. We will show that $e_\Delta(t)=f(t-1)=(-1)^df_\Delta(-t)$. Given a simplex $\sigma\in\Delta$, let $|\sigma|$ denote the number of vertices in $\sigma$. We make the following two observations:

$$\sum_{\sigma\subset\tau}(-1)^{|\tau|}=(-1)^d \text{ and } (1-t)^{|\sigma|}=\sum_{\tau\subset\sigma} (-t)^{|\tau|}.$$

The first equation follows from the fact that $\Delta$ is a Eulerian sphere and the second equation is just the Binomial Theorem. We compute:

$$f(t-1)=\sum_{i=0}^d f_{i-1}(t-1)^i=\sum_{\sigma\in\Delta} (t-1)^{|\sigma|}=\sum_{\sigma\in\Delta}(-1)^{|\sigma|}(1-t)^{|\sigma|}.$$

We now apply the two observations above to the right-most sum:

$$f(t-1)=\sum_{\sigma\in\Delta}\sum_{\tau\subset\sigma} (-t)^{|\tau|}(-1)^{|\sigma|}=(-1)^d\sum_{\sigma\in\Delta}\sum_{\tau\subset\sigma} (-t)^{|\tau|}=(-1)^d\sum_{\tau\in\Delta}(-t)^{|\tau|}.$$

Note that the right-most sum is just equal to $(-1)^df(-t)$. This completes the proof for the case when $\Delta$ is a Eulerian sphere. For the last part of the assertion, note that if $\Delta$ is an odd-dimensional Eulerian manifold, then \cite[Lemma 17.3.3]{Dav} implies that $\tilde{\chi}(\Delta)=0$. So, an odd-dimensional Eulerian manifold $\Delta$ is an Eulerian sphere, and hence has Property E.
\end{proof}

\begin{remark*}
If $\Delta$ is Eulerian manifold, then by \cite[Theorem 1]{Aki}, $\tilde{\chi}(\Delta)=f_\Delta\left(-\frac{1}{2}\right)=e_\Delta\left(\frac{1}{2}\right)$. If $\Delta$ is odd-dimensional, then $\tilde{\chi}(\Delta)=0$. It follows that $t=\frac{1}{2}$ is a root of $e_\Delta(t)$, and hence $t=-\log 2$ is a root of the exponential Hilbert series.
\end{remark*}

The following example completely characterizes $1$-dimensional complexes with Property E. It also provides non-Eulerian examples satisfying Property E, thus showing that the converse of Theorem 5 is false.
\vspace{1pc}
\\
\noindent\textbf{Example:} Suppose that $\Delta$ is $1$-dimensional and connected. Then $e_0=-\chi(\Delta)=1-f_0+f_1$, $e_1=f_0-2f_1$ and $e_2=f_1$. If $\Delta$ has weak Property E, then it follows that $e_1=-f_0$, and therefore $f_0=f_1$. This means that $\chi(\Delta)=-1$, and hence $\tilde{\chi}(\Delta)=0$. So for $1$-dimensional $\Delta$, Property E and weak Property E are equivalent.

If $\tilde{\chi}(\Delta)=0$, then it is homotopic to a cycle. Any such $\Delta$ looks like a cycle with trees attached to its vertices, which clearly satisfies $f_0=f_1$. It follows that a connected $1$-dimensional $\Delta$ satisfies Property E if and only if it is a cycle with trees emanating from its vertices.

Consider such a $\Delta$ which is an empty triangle with an extra edge attached to one of its vertices. This $\Delta$ has Property E, but is not Eulerian, as the links of the vertices of the extra edge do not have the correct Euler characteristic.
\vspace{1pc}
\\
\indent The following proposition allows one to produce many more examples of complexes with Property E. In particular, it will allow us to exploit the previous example to produce non-Eulerian examples satisfying Property E in all dimensions.

\begin{prop}
Suppose that $\Delta=\Delta_1\ast\Delta_2$ is the join of two simplicial complexes $\Delta_1$ and $\Delta_2$ satisfying Property E. Then $\Delta$ has Property E.
\end{prop}
\begin{proof}
We have that $f_\Delta(t)=f_{\Delta_1}(t)f_{\Delta_2}(t)$. As $e_\Delta(t)=f_\Delta(t-1)$, it follows that $e_\Delta(t)=e_{\Delta_1}(t)e_{\Delta_2}(t)$. Suppose that $\Delta_1$ is $l$-dimensional and $\Delta_2$ is $m$-dimensional (so that $\Delta$ is $l+m+1$-dimensional). As both $\Delta_1$ and $\Delta_2$ have Property E, we have that $e_{\Delta_1}(t)=(-1)^{l+1}f_{\Delta_1}(-t)$ and $e_{\Delta_2}(t)=(-1)^{m+1}f_{\Delta_2}(-t)$. It follows that $e_\Delta(t)=(-1)^{l+m+2}f_{\Delta_1}(-t)f_{\Delta_2}(-t)=(-1)^{l+m+2}f_{\Delta}(-t)$, and therefore $\Delta$ has Property E.
\end{proof}

By taking iterated suspensions of the complexes in the previous example, the above proposition allows one to produce non-Eulerian complexes satisfying Property E in every dimension. Hence (via Theorem 3) there are many non-Eulerian complexes satisfying the classical Dehn-Sommerville relations. A characterization of simplicial complexes with (weak) Property E would be interesting.

We conclude with a generalization of Theorems 3 and 5. First, we introduce the notion of an Eulerian manifold with boundary. A pure $(d-1)$-simplicial complex $\Delta$ is an \textit{Eulerian manifold with boundary} $\partial \Delta$ if the following conditions hold:

\begin{description}
  \item[(i)] $\partial\Delta$ is a subcomplex of $\Delta$ and $\partial\Delta$ is itself a $(d-2)$-Eulerian manifold.
  \item[(ii)] Every simplex $\sigma\in\Delta$ not contained in $\partial\Delta$ has the property that $\tilde\chi(\Lk \sigma)=1+(-1)^{d+\dim\sigma}$.
  \item[(iii)] Every simplex (except the empty simplex) $\sigma\in \partial\Delta$ has the property that $\tilde\chi(\Lk \sigma)=1$.
\end{description}

We note that in condition (iii) we are taking the link of the simplex in the entirety of $\Delta$. This definition is modeled on triangulations of manifolds with boundary. While the concept is not entirely new, it has only been considered in \cite{CL} from the perspective of posets (there it is called a semi-Eulerian partially ordered set with boundary).

In \cite{NS}, the following version of the Dehn-Sommerville relations is proved for homology manifolds with boundary. However, the proof goes through line-by-line for Eulerian manifolds with boundary.

\begin{theorem}[{\cite{NS}[Theorem 3.1]}]
If $(d-1)$-simplicial complex $\Delta$ is an Eulerian manifold with boundary $\partial\Delta$ then for $0\leq k\leq d$ we have that $$h_k(\Delta)-h_{d-k}(\Delta)= {d\choose k}(-1)^{d-k}\tilde{\chi}(\Delta)+h_k(\partial\Delta)-h_{k-1}(\partial\Delta).$$
\end{theorem}

In terms of the $e$-vector, this reads as follows.

\begin{theorem}
The Dehn-Sommerville relations for Eulerian manifolds with boundary are equivalent to the following condition on the $e$-vector:

$$e_k(\Delta)=(-1)^{d-k} f_{k-1}(\Delta)+e_k(\partial\Delta)\text{ for } 1\leq k\leq d.$$
\end{theorem}
\begin{proof}
The proof is in the same vein as of that of Theorem 3. In terms of $h$-polynomials, the Dehn-Sommerville relations of Theorem 7 can be rewritten as:

$$h_\Delta(t)-t^d h_\Delta\left(\frac{1}{t}\right)=(-1)^d\tilde{\chi}(\Delta)K(t)+(1-t)h_{\partial\Delta}(t),$$

where $$K(t)=\sum_{k=0}^d (-1)^k {d \choose k} t^k. $$

Equivalently,

$$e_\Delta\left(\frac{1}{1-t}\right)+(-1)^{d+1} e_\Delta\left(\frac{t}{t-1}\right)=(-1)^d\tilde{\chi}(\Delta)\cdot \frac{K(t)}{(1-t)^d}+e_{\partial\Delta}\left(\frac{1}{1-t}\right).$$

Making the substitution $t=1-\frac{1}{u}$, we obtain $$e_\Delta(u)=(-1)^d\tilde{\chi}(\Delta)+(-1)^{d} e_\Delta(1-u)+e_{\partial\Delta}(u).$$

The result follows by noting that $e_\Delta(1-u)=f_\Delta(-u)$.

\end{proof}

\section{Hierarchical log-linear models}

As an application of the $e$-vector, we consider hierarchical log-linear models from the point of view of algebraic statistics. To that end, let $X=(X_1,\dots,X_m)$ be a discrete random vector, and assume that each $X_i$ has state space $[r_i]$. Let
\begin{center}
$\mathcal{R}:=\displaystyle\prod_{i=1}^m[r_i]$
\end{center}
denote the joint state space of the random vector $X$.

We use the convention in \cite{Sull} for writing subindices: let $i=(i_1,\dots,i_m)\in\mathcal{R}$ and $\sigma=\{f_1,f_2,\dots\}\subseteq[m]$. Then we define
\begin{center}
$i_\sigma:=(i_{f_1},i_{f_2},\dots)$.
\end{center}
For each $\sigma\subseteq[m]$, the random subvector $X_\sigma:=(X_f)_{f\in \sigma}$ has state space $\mathcal{R}_\sigma=\displaystyle\prod_{f\in \sigma}[r_f]$. With this notation set, we define a \textit{hierarchical log-linear model} as follows (this is Definition 9.3.2 in \cite{Sull}).

\begin{defn}
Let $\Delta\subseteq 2^{[m]}$ be a simplicial complex and let $r_1,\dots,r_m\in\mathbb{N}$. For each facet $\sigma\in\Delta$, let $\theta^{(\sigma)}_{i_\sigma}$ be a set of $\#\mathcal{R}_F$ \textit{positive} parameters. Then the \textbf{hierarchical log-linear model} associated with $\Delta$ is the set of all probability distributions
\begin{center}
$\mathcal{M}_\Delta:=\left\{p\in\Delta_{\mathcal{R}-1} : p_i=\displaystyle\frac{1}{Z(\theta)}\prod_{\sigma\subset\Delta}\theta^{(\sigma)}_{i_\sigma}, \forall i\in\mathcal{R}\right\}$,
\end{center}
where $Z(\theta)$ is the normalizing constant
\begin{center}
$Z(\theta)=\displaystyle\sum_{i\in\mathcal{R}}\prod_{\sigma\subset\Delta}\theta^{(\sigma)}_{i_\sigma}$.
\end{center}
\end{defn}

Log-linear models in general can be described in terms of integer matrices that record the non-trivial interactions between the variables. We will denote by $A_\Delta$ the integer matrix that describes the hierarchical model. Note that $A_\Delta$ is a $\{0,1\}$-matrix, as monomials corresponding to facets must be squarefree, i.e. variables only interact with degree 1. The rank of $A_\Delta$ was first computed in \cite{HS}. It encodes the dimension (and hence the number of degrees of freedom in the probability simplex) of the model.

\begin{prop}[Ho\c{s}ten-Sullivant]
Let $\Delta$ be a simplicial complex on $[m]$, and $r_1,\dots,r_m\in\mathbb{N}$. The rank of the matrix $A_\Delta$ associated to these parameters is
\begin{center}
$\displaystyle\sum_{F\in\Delta}\prod_{f\in F}(r_f-1)$,
\end{center}
where the sum runs over all faces of $\Delta$. The dimension of the associated hierarchical model $\mathcal{M}_\Delta$ is one less than the rank of $A_\Delta$.
\end{prop}

The original proof of this proposition was done by studying the kernel of $A_\Delta$. In other words, the proof is based on the structure of the matrix as a linear map.

\begin{prop}
Let $X=(X_1,\dots,X_m)$ be a random vector with $X_i$ having $r_i$ possible outcomes. Let $\Delta$ be an abstract simplicial complex and let $\mathcal{M}_{A_{\Delta}}$ be the hierarchical log-linear for $X$ on $\Delta$. Then the rank of the matrix $A_\Delta$ of sufficient statistics is
\begin{center}
$E(\Delta;\log(r_1),\dots,\log(r_m))$.
\end{center}
\end{prop}

\begin{proof}
The proof of this result is really just a combination of two previous theorems. Recall that,
\begin{equation}
E(\Delta;\textbf{x})=\displaystyle\sum_{\sigma\subseteq\Delta}\prod_{i\in \sigma}(e^{x_i}-1).
\end{equation}
We also have, from \cite{Sull},
\begin{equation}
\text{rank}(A_\Delta)=\displaystyle\sum_{\sigma\in\Delta}\prod_{i\in \sigma}(r_i-1).
\end{equation}
Then substituting $x_i\mapsto\log(r_i)$ into (4.1) yields (4.2).
\end{proof}

\begin{prop}
Assume that $r_1=r_2=\dots=r_m=r$. Then the rank of $A_\Delta$ is
\begin{center}
$E(\Delta;\log(r))=e_0+e_1r+\dots+e_dr^d$,
\end{center}
where $d=\dim(\Delta)$.
\end{prop}

\begin{proof}
Let $r_1=r_2=\dots=r_m=r$. Then,
\begin{center}
$\text{rank}(A_\Delta)=E(\Delta;\log(r),\dots,\log(r))$.
\end{center}
Plugging $\log(r)$ in for each variable in Theorem 2 yields
\begin{center}
$E(\Delta;\log(r),\dots,\log(r))=\displaystyle\sum_{\sigma\subset\Delta}\prod_{i\in \sigma}(e^{\log(r)}-1)=\sum_{\sigma\subset\Delta}\prod_{i\in \sigma}(r-1)$.
\end{center}
Expanding the right-hand side over the product yields
\begin{center}
$\displaystyle\sum_{\sigma\subset\Delta}(r-1)^{\dim(\sigma)+1}=E(\Delta;\log(r))$.
\end{center}
\end{proof}

Proposition 11 states that, as long as each random variable in $X$ has the same number of outcomes, the dimension is computed via the $e$-vector of $\Delta$.

As a quick application of Proposition 11, we consider cyclic models. In a \textit{cyclic} model on $[m]$, the maximal faces of $\Delta$ are $\{1,2\},\{2,3\}, \dots, \{(m-1),m\}, \{m,1\}$. Thus, the face vector of $\Delta$ is $(1,m,m)$, as $\Delta$ is a graph with $m$ vertices and $m$ edges. The $e$-vector of $\Delta$ is easily computed to be $(1,-m,m)$. Note that the cyclic model follows the Dehn-Sommerville relations, as can be seen in the $e$-vector.

\begin{cor}
Let $\Delta$ be the simplicial complex of the cyclic model, and assume $r_1=r_2=\dots=r_m=r$. Then
\begin{center}
$\text{rank}(A_\Delta)=1-mr+mr^2$.
\end{center}
\end{cor}

As a final example, let $\Delta$ be the boundary of the saturated model. The saturated model is the model whose complex is the power set on $[m]$. Therefore $\Delta$ is the model that includes all interaction factors except for the interaction term between all variables. $\Delta$ has the same number of $i$-faces as the saturated model, except $\Delta$ has no $m-1$ face. Since $\Delta$ is the boundary of a simplicial polytope, it satisfies the Dehn-Sommerville relations. Therefore, we have
\begin{center}
$e_k=(-1)^{k+1}f_{k-1}$,
\end{center}
for $k=0,\dots,m-1$, which yields
\begin{center}
$\text{rank}(A_\Delta)=\displaystyle\sum_{i=0}^{m-1}(-1)^{i+1}{m\choose i}r^i$,
\end{center}
assuming each random vector has $r$ outcomes.

\end{document}